\documentclass[12pt]{amsart}
\usepackage{amssymb,amsmath,amsthm,bm}
\usepackage[colorlinks=true,allcolors=blue]{hyperref}
\usepackage[charter,cal=cmcal]{mathdesign}
\DeclareSymbolFont{usualmathcal}{OMS}{cmsy}{m}{n}
\DeclareSymbolFontAlphabet{\mathcal}{usualmathcal}

\usepackage[T1]{fontenc}

\usepackage{color}

\def\cic{\bm}

\def\d{{\rm d}}

\def\R {\mathbb{R}}

\def\G{{\mathcal G}}

\def\D {{\mathcal D}}

\def \l {\langle}
\def \r {\rangle}

\def \and{\qquad\text{and}\qquad}

\newcommand{\supp}{\mathrm{supp}\,}

\newcounter{thms}
\newcounter{other}
\numberwithin{other}{section}
\newtheorem{proposition}[other]{Proposition}
\newtheorem{theorem}[thms]{Theorem}
\newtheorem*{theorem*}{Theorem}
\newtheorem*{proposition*}{Proposition}
\newtheorem{corollary}[thms]{Corollary}
\newtheorem{lemma}[other]{Lemma}
\theoremstyle{definition}
\newtheorem{remark}[other]{Remark}

\numberwithin{equation}{section}

\title[Sparse domination via dyadic shifts]{Uniform   sparse domination of singular integrals \\  via dyadic shifts}

 \author{Amalia Culiuc}
 \address{\noindent School of Mathematics, Georgia Institute of Technology, \newline \indent Atlanta, GA 30332, USA (A.\ Culiuc)}
 \author{Francesco Di Plinio} \address{\noindent Department of Mathematics, University of Virginia,  \newline \indent Kerchof Hall,  Box 400137, Charlottesville, VA 22904-4137, USA  (F.\ Di Plinio)}
 \author{Yumeng Ou}
 \address{\noindent Department of Mathematics, Massachusetts Institute of Technology, \newline \indent  77 Massachusetts Avenue, Cambridge, MA 02139, USA (Y. Ou)\vskip2mm }
\email[A.\ Culiuc]{amalia@math.gatech.edu}
\email[F.\ Di Plinio]{francesco.diplinio@virginia.edu}
\email[Y.\ Ou]{yumengou@mit.edu}

 \subjclass[2010]{Primary: 42B20. Secondary: 42B25}
 \keywords{Positive sparse operators,   $T(1)$ theorems, weighted norm inequalities, matrix weights}
\thanks{F. Di Plinio was partially
supported by the National Science Foundation under the grants
   NSF-DMS-1500449 and  NSF-DMS-1650810.}

\begin{document}
 \begin{abstract} Using the Calder\'on-Zygmund decomposition, we give a novel and  simple proof that $L^2$ bound\-ed dyadic shifts admit a domination by positive sparse forms with linear growth in  the complexity of the shift. Our estimate, coupled with Hyt\"onen's dyadic representation theorem, upgrades to a positive sparse domination of the class $\mathcal U$ of singular integrals satisfying the assumptions of the classical $T(1)$-theorem of David and Journ\'e.
 Furthermore, our  proof extends rather easily to the $\R^n$-valued case, yielding as a corollary  the operator norm bound on the matrix weighted  space $L^2(W; \R^n)$
 \[
 \left\|T\otimes \mathrm{Id}_{\R^n}\right\|_{L^2(W; \R^n)\rightarrow L^2(W; \R^n)} \lesssim [W]_{A_2}^{\frac32}
 \]
 uniformly over $T\in \mathcal U$, which is the currently best known dependence.
  \end{abstract}
\maketitle
\section{Main results and context} \label{s1}

Set in motion by the seminal article of Andrei Lerner \cite{Ler2013}, the pointwise control of singular integral operators by positive sparse averages of the input functions has proved to be a remarkably effective strategy towards sharp weighted norm inequalities, within and beyond Calder\'on-Zygmund theory.

In this note, we set forth a novel and simple approach to positive sparse domination of singular integral operators, at the core of which lies the classical Calder\'on-Zygmund decomposition. Our approach has the advantage of extending rather effortlessly to the case of singular integrals acting on $\R^n$-valued functions, thus yielding matrix weighted norm inequalities with quantified dependence on the matrix weight characteristic. We   provide additional context after the statement of our main results.
 
\subsection{Main results}Our first domination result, Theorem  \ref{domshift}, involves  dyadic shifts, which are the fundamental discrete model for  Calder\'on-Zygmund operators. We send to   Hyt\"onen \cite{HytA2} and references therein for more precise versions of this statement, and proceed to the formal  definition.   Let $\mathcal D$ be a dyadic lattice on $\R^d$, and $m_1,m_2$ be two nonnegative integers. A bilinear form
\[
\mathbb S^\varrho(f_1,f_2) = \sum_{Q \in \mathcal D} S_Q(f_1,f_2), \qquad  S_Q(f_1,f_2):=\int_{Q\times Q} s_Q(x_1,x_2) f_1(x_1)  {f_2(x_2)} \, \d x_1 \d x_2
\]
defined for $f_j\in L^1_{\mathrm{loc}} (\R^d)$, $j=1,2$ is termed a \emph{dyadic shift} of complexity $\varrho=\max\{1,m_1,m_2\}$ constructed on $\mathcal D$ if the following assumptions hold:
\begin{itemize}
\item[A1.] $s_Q: Q\times Q\to \mathbb C$ is  zero for all but finitely many $Q\in \mathcal D$ and $\|s_Q\|_\infty \leq |Q|^{-1}$.
\item[A2.] 
there holds \[\sup_{\mathcal Q\subset \mathcal D}|\mathbb S^\varrho_{\mathcal Q}(f_1,f_2)| \leq \|f_1\|_2\|f_2\|_2\] 
where the \emph{subshifts}  $\mathbb S^\varrho_{\mathcal Q}$ are defined by
\[
\mathbb S^\varrho_{\mathcal Q}(f_1,f_2): = \sum_{Q \in \mathcal Q}  S_Q(f_1,f_2).
\]
\item[A3.] if $R\in \mathcal D, R\subset Q $ and $\ell(R) <2^{-m_1}\ell(Q)$ 
then $ s_Q(\cdot,x_2)$ is constant on $R$ for all $x_2 \in Q$, and symmetric assumption with the roles of $x_1,x_2$ interchanged.\end{itemize}
We now introduce the ingredients of a positive sparse form.  For a cube $Q\subset \R^d$, we write
\[
\l f \r_Q := \frac{1}{|Q|} \int_Q |f| \, \d x.
\]
We say that a collection $  \mathcal S$ of cubes of $\R^d$  is $\eta$-sparse if for each $Q\in \mathcal S$ there exists $E_Q\subset Q$ with $|E_Q|\geq \eta |Q|$ and such that the sets $\{E_Q: Q \in \mathcal S\}$ are pairwise disjoint. The precise value of $\eta<1$ will be of no interest for us in what follows.
\begin{theorem} \label{domshift} There exists an absolute constant $C$ such that the following holds. For every $f_1,f_2 \in L^1(\R^d)$  with compact support and every dyadic lattice $\mathcal D$  there exists a sparse collection $\mathcal S_\D$ such that for all $\varrho\geq 1$\begin{equation}
\label{unifest}
\sup_{   \mathbb S^\varrho} |\mathbb S^\varrho(f_1,f_2)| \leq C \varrho    \sum_{Q \in \mathcal S_\D} |Q| \l f_1 \r_{Q} \l f_2 \r_{Q}  
\end{equation}
the supremum being taken  over all  {dyadic shifts} $ \mathbb S^\varrho  $ of complexity $\varrho$ constructed on $\mathcal D$.
\end{theorem} The proof   is given in Section \ref{spf}.
We may upgrade the   sparse domination of Theorem \ref{domshift}  to a domination of singular integrals satisfying smoothness assumptions on the kernel and David-Journ\'e type testing conditions, by using the well-known representation principle of Hyt\"onen \cite{HytA2}, built upon previous work of Nazarov, Treil and Volberg \cite{NTV03}. 

Let $\mathcal U$ be the family of singular integral operators, acting on a  dense subspace  $\mathcal W$ of $L^2(\R^d)$ containing, say, bounded functions with compact support,  and satisfying the following quantitative assumptions.
\begin{itemize}
\item[B1.]
Each $T\in \mathcal U$ has kernel representation
\[
Tf(x) = \int_{\R^d} K(x,y) f(y) \, \d y, \qquad x\not \in \supp f
\]
with $K: \R^d\times \R^d\backslash\{(x,y):x=y\}\to \mathbb C$ satisfying  the standard estimates
\[
\begin{split}
&|K(x,y)| \leq \frac{1}{|x-y|^d},\\
& |K(x+h,y)-K(x,y)|+|K(x,y+h)-K(x,y)| \leq \frac{\omega\left(\textstyle\frac{|h|}{|x-y|} \right)}{|x-y|^d}  \qquad \forall |h|<\frac{|x-y|}{2}
\end{split}
\]
where $\omega$ is the modulus of continuity $\omega(t)=t^\alpha$, $t\in (0,1]$, and $\alpha\in (0,1]$ is fixed. 
\item[B2.] There holds
\[
\sup_{Q} |\l T(\cic{1}_Q),\cic{1}_Q \r|   \leq |Q|
\] 
the supremum being taken over cubes $Q\subset \R^d$.
\item[B3.] With an appropriate definition of $T\cic{1}, T^*\cic{1}$, there holds
\[
\|T\cic{1}\|_{\mathrm{BMO}(\R^d)}, \,\|T^*\cic{1}\|_{\mathrm{BMO}(\R^d)} \leq 1.
\]
\end{itemize}
The following proposition is a restatement of the representation theorem from \cite{HytA2}, in the more precise version provided in \cite{HytLN}.
\begin{proposition}{\cite[Theorem 3.3]{HytLN}} \label{rt} Let $f_1,f_2\in \mathcal W$. There holds 
\begin{equation} \label{propeq}
\sup_{T\in \mathcal U}| \l T f_1, \overline {f_2}\r |\leq C \sup_{\varrho \geq 1} \sup_{\mathcal D,\mathbb S^\varrho} \varrho^{-1}  |\mathbb S^\varrho(f_1,f_2)| 
\end{equation}
with a constant $C>0$ depending on $d, \alpha$ only, the second supremum being taken over all dyadic lattices $\mathcal D$ of $\R^d$ and all  {dyadic shifts} $ \mathbb S^\varrho  $ of complexity $\varrho$ constructed on $\mathcal D$.
\end{proposition}
\begin{remark} In \cite{HytLN,HytA2}, following ideas of \cite{NTV03}, the author constructs  a family  $\mathcal D_\omega$ of dyadic lattices parametrized by $\omega\in\Omega=(\{0,1\}^d)^{\mathbb Z}$. Then, Theorem 3.3 of \cite{HytLN} rewritten in our language yields that for each $T \in \mathcal U$, $f_1,f_2 \in \mathcal W$, the equality  
\begin{equation} \label{hyteq}
\l Tf_1,\overline{f_2}\r= \mathbb{E}_{\omega} \sum_{\varrho=1}^\infty \tau(\varrho)   \mathbb{S}^{\varrho}_{\omega}(f_1,f_2)
\end{equation}  holds with a suitable choice  of  $\mathbb{S}^{\varrho}_{\omega}$,  a dyadic shift of complexity $\varrho$ constructed on the dyadic lattice $\mathcal D_\omega$, and with a  
sequence
$\{\tau(\varrho)\}$ satisfying \[|\tau(\varrho)| \leq C 2^{-\frac{\alpha}{2} \varrho},\]
the expectation in \eqref{hyteq} being taken over the natural probability measure on $ \Omega.$
The uniform estimate of Proposition \ref{rt} thus follows by dominating the right hand side of \eqref{hyteq} by the right hand side of \eqref{propeq} times the series of   $\varrho2^{-\frac{\alpha}{2} \varrho}$, and summing the series.
\end{remark}
\begin{remark}
We note that more refined versions of the representation formula of \cite{HytLN} may be employed to extend \eqref{propeq} to logarithmic-type moduli of continuity $\omega$; see  for instance the very recent article \cite{GH}. However, these methods fall short of tackling the Dini-continuous case first settled in \cite{Lac2015}. For this reason, and given that the main aim of this paper is to present a new sparse domination technique in the simplest possible setting, we choose to restrict our analysis to power-type moduli of continuity.  
\end{remark}
Coupling the domination Theorem \ref{domshift} with Proposition \ref{rt} yields the following   sparse domination theorem.
\begin{theorem} \label{mainthm}  There exists an absolute constant $C>0$ such that the following holds. For every $f_1,f_2 \in \mathcal W$ and having compact support  there exists a sparse collection   $\mathcal S$ such that
\[
\sup_{T\in \mathcal U} |  \l T f_1, \overline {f_2} \r| \leq C  \sum_{Q \in \mathcal S} |Q| \l f_1 \r_{Q} \l f_2 \r_{Q}.  
\]
\end{theorem}
\begin{proof}
Fix a pair of functions $f_1,f_2 \in \mathcal W$  with compact support. A combination of Theorem \ref{domshift} with Proposition \ref{rt} readily yields the inequality
\[
\sup_{T\in \mathcal U} |  \l T f_1, \overline {f_2} \r| \leq  C  \sup_{\mathcal T} \sum_{Q \in \mathcal T} |Q| \l f_1 \r_{Q} \l f_2 \r_{Q},\]
the supremum being taken over \emph{all} sparse collections  $\mathcal T$. 
 The proof of Theorem \ref{mainthm} is then finished by the observation that there exists a sparse collection $\mathcal S$ (depending only on $f_1,f_2$) such that
\[
\sup_{\mathcal T} \sum_{Q \in \mathcal T} |Q| \l f_1 \r_{Q} \l f_2 \r_{Q} \leq  C\sum_{Q \in \mathcal S} |Q| \l f_1 \r_{Q} \l f_2 \r_{Q} ,
\]
The last claim follows via a simple stopping time argument based on the size of $\l f_1 \r_{Q} \l f_2 \r_{Q}$; we send to  \cite[Lemma 4.7]{LMena} for the full proof.
\end{proof}
Our proof of the dyadic shift domination Theorem \ref{domshift}   is based on a stopping time argument akin to the one employed by the authors in \cite{CuDPOu} to prove a uniform sparse domination theorem for  the bilinear Hilbert transforms. At the heart of both lies a Calder\'on-Zygmund decomposition: classical, in Lemma \ref{mainiter} of this paper, around multiple frequencies in the outer $L^p$-embedding theorem \cite{DPOu2} by two of us which is relied upon in  \cite{CuDPOu}.
 
{{}{Perhaps surprisingly, our approach  extends effortlessly to singular integrals acting on functions taking values in a finite-dimensional Euclidean space, once a suitable vector valued version of the  Calder\'on-Zygmund decomposition is introduced with Lemma \ref{vectorstop}. In Section \ref{s3}, we adapt our proof of Theorem \ref{domshift} to obtain uniform positive sparse domination of singular integrals in the class $\mathcal U$: see Theorem \ref{mainthmvec}. Besides its intrinsic interest, Theorem \ref{mainthmvec}  also yields the currently best known quantitative matrix $A_2$ weighted estimates for the $\R^n$-valued extension of operators of the class $\mathcal U$: see Corollary \ref{corweight}. Our positive sparse forms in this setting involve the Minkowski product of convex sets generated by local averages of the input functions, a variation on a theme proposed by  Nazarov, Petermichl, Treil and Volberg \cite{NTVc}. 
\subsection{Context}
We turn to a deeper description of the context and consequences of our approach. The pointwise control of a  Calder\'on-Zygmund operator $T$ by $2^d$ sparse averaging operators depending on $T$ itself and on the input function $f$, 
\begin{equation}
\label{pdom}
|Tf(x)| \leq C \sum_{1\leq j\leq 2^d} \sum_{Q \in \mathcal S_j}  \l f \r_{Q}\cic{1}_{Q}(x)
\end{equation}
has been first achieved independently by Conde-Alonso and Rey \cite{CR} and Lerner and Nazarov \cite{LerNaz2015}, elaborating on Lerner's seminal paper \cite{Ler2013}. A powerful approach to \eqref{pdom} forgoing the local mean oscillation estimate has been introduced by Lacey in \cite{Lac2015} and subsequently streamlined by Lerner \cite{Ler2015}. In contrast to all these previous works,    the weak $(1,1)$ estimate for $T$ is not an \emph{a priori} assumption of our Theorem  \ref{mainthm}. Rather, it is obtained as a consequence of the domination theorem from the standard assumptions of a $T(1)$ theorem. Furthermore, the sparse collection in Theorem \ref{mainthm} is explicitly constructed from level sets of the maximal function rather than the specific operator, and is thus the same for all operators in the class $\mathcal U$.}}

 We also note that   than the domination by sparse forms as in Theorem \ref{mainthm}, while formally weaker, seems to be just as useful as the pointwise control \eqref{pdom}. In fact, it is by dualizing \eqref{pdom} that the essential disjointness of $Q\in \mathcal S_j$ may be exploited. Hence, just as well as \eqref{pdom}, Theorem \ref{mainthm} leads rather immediately to  Hyt\"onen's sharp weighted inequalities \cite{HytA2}
\begin{equation}
\label{A2thm}
\sup_{T \in \mathcal U} \|T\|_{L^p(w)\rightarrow L^p(w)} \leq C\textstyle \frac{p}{p-1} [w]_{A_p}^{\max\{1,\frac{1}{p-1}\}}.
\end{equation}
See \cite{LiMoen2014} for a self-contained argument deducing \eqref{A2thm} from Theorem \ref{mainthm}. A more general treatment  is provided in \cite[Section 16]{LerNaz2015}. 
Comparing to the routes to the $A_2$-theorem   outlined in the interesting survey  \cite{HytSurv},  we believe that our approach provides an additional shortcut to a sharp weighted $T(1)$ theorem stemming directly from the representation theorem of \cite{HytA2}.  It is likely that our proof strategy may be  further applicable within the developing field of sparse domination in the nonhomogeneous and noncommutative setting: see \cite{CAP} for a recent breakthrough result.

{ Concerning the vector-valued extension, quantified matrix $A_2$ estimates have appeared in the recent works  \cite{BCTW,BPW,BW0}. A closely related result to Theorem \ref{mainthmvec}, involving the Minkowski sum of convex body-valued sparse operators rather than bilinear forms, was announced by Nazarov, Petermichl, Treil and Volberg \cite{NTVc} before the present article was prepared. The details of their argument were unknown to us at the time of completion of the first version of this article, and were made public in the preprint \cite{NPTV}, while our own article was being revised for publication.}

\subsection*{Acknowledgements} The concept of domination by convex body averages originates from an idea of Fedor Nazarov. The authors want to thank Sergei Treil for introducing this idea to them during his seminar talk at Brown University in the Spring of 2016. The authors also extend their gratitude to Michael Lacey, Jill Pipher and  Brett Wick for their comments on an early version of this manuscript, and to the anonymous referee for the valuable suggestions.

\section{Proof of Theorem \ref{domshift}} \label{spf}   Throughout this proof, we denote by $C$ a positive   constant which is allowed to depend on the dimension $d$ only and whose value may vary from line to line without explicit mention.

 \subsection{Construction of the sparse collection $\mathcal S_\D(f_1,f_2)$}Below, we write $\mathcal D(Q):=\{R\in \mathcal D: R\subset Q\}$.
For $f_1,f_2\in L^1_{\mathrm{loc}}(\R^d)$ and  $Q\in \mathcal D$, we define $I \in\mathcal I_Q$ to be the maximal elements of $\D(Q)$ such that
 \begin{equation}
\label{stop}
\l f_j \r_{I}> 2^{8} \l f_j \r_{Q}\qquad  \textrm{{for at least one }} j=1,2.
\end{equation}
Then
 \begin{equation}
\label{sparsepack}
 \sum_{I \in \mathcal I_Q} |I|\leq 2^{-7}|Q|.
\end{equation}
Now, if   $f_1,f_2\in L^1(\R^d)$  with compact support we may find $2^d$ adjacent congruent cubes $Q_1,\ldots Q_{2^d}\in \D$, whose union is $\overline Q$, such that ${\overline Q}$ contains the support of both $f_1,f_2$.
For each $\ell=1,\ldots,2^d$,  referring to the definition \eqref{stop},  we   inductively set
\[
\mathcal S_{\D,\ell,0}:=\{Q_\ell\}, \qquad \mathcal{S}_{\D,\ell,n}:= \bigcup_{Q \in \mathcal{S}_{\D,\ell,n-1}} \mathcal I_Q, \; n=1,2,\ldots \qquad  \mathcal{S}_{\D,\ell}:=\bigcup_{n\geq 1} \mathcal{S}_{\D,\ell,n}.
\]
By using the packing estimate \eqref{sparsepack} and disjointness of $Q_\ell,$ it is easy to see that
\[\mathcal S_{\D}:= \{\overline Q\} \cup \bigcup_{\ell=1}^{2^{d}} \mathcal{S}_{\D,\ell}\]
is a sparse  collection of dyadic cubes.
The reason for  employing  the larger $\overline Q$ in place of each $Q_\ell$ will be clear below.\subsection{Main line of proof of \eqref{unifest}}
The main  thrust of the proof    is provided by the lemma below, which we plan to apply iteratively.
\begin{lemma} \label{mainiter} Let $f_1,f_2\in L^1_{\mathrm{loc}}(\R^d)$ and  $Q\in \mathcal D$.  For all dyadic shifts  $\mathbb S^\varrho$ constructed on $\mathcal D$
\[
|\mathbb S^\varrho(f_1\cic{1}_Q,f_2\cic{1}_Q)| \leq C \varrho |Q| \l f_1 \r_Q \l f_2 \r_Q + \sum_{I \in \mathcal I_Q} \big|\mathbb S^\varrho_{\mathcal D(I)}(f_1\cic{1}_I ,f_2\cic{1}_I )\big| 
\]
\end{lemma}
Fixing $f_1,f_2\in L^1(\R^d)$ with compact support and having constructed $\mathcal S_\D(f_1,f_2)$ in the previous subsection, we turn to the proof of \eqref{unifest} assuming the lemma. Let $\mathbb S^\varrho$ be a fixed but arbitrary dyadic shift of complexity $\varrho$ constructed on $\mathcal D$. We expand 
\begin{equation}
\label{thesum}
\mathbb S^\varrho(f_1,f_2)=\mathbb S^\varrho(f_1 \cic{1}_{\overline Q},f_2\cic{1}_{\overline Q})=\sum_{k,\ell=1}^{2^d}\mathbb S^\varrho(f_1 \cic{1}_{ Q_k},f_2\cic{1}_{  Q_\ell}) 
\end{equation}
 and  estimate the terms $\mathbb S^\varrho(f_1 \cic{1}_{ Q_k},f_2\cic{1}_{  Q_\ell})$ for $k\neq \ell$. We have for any $R\subset \mathcal{D}$ that
\[
S_{R}(f_1 \cic{1}_{ Q_k},f_2\cic{1}_{  Q_\ell})\begin{cases} = 0 & \textrm{if either } Q_k \not\subseteq R   \textrm{ or } Q_\ell \not\subseteq R \\  = S_{R}(\l f_1 \r_{ Q_k}\cic{1}_{Q_k}, \l f_2 \r_{ Q_\ell} \cic{1}_{Q_\ell} ) & \textrm{if   } Q_k, Q_\ell \subset R, \ell(R) \geq 2^{\varrho} \ell(Q_\ell)  \\ \textrm{bounded in abs.\ value by }|Q_k| \l f_1 \r_{ Q_k} \l f_2 \r_{ Q_\ell} & \textrm{otherwise}.
\end{cases}
\]
The second condition is satisfied since  the kernel $s_R$ of $S_R$ is such that
\[
y\mapsto s_R(x_0,y), \qquad x\mapsto s_R(x,y_0)
\]
are constant on $Q_k,Q_\ell$ respectively for all $x_0,y_0$.
There are at most $\varrho$ ``otherwise'' cases. Thus  using the  $L^2$-bound  on the second summand in the first right hand side,
\[
\begin{split} &\quad 
|\mathbb S^\varrho(f_1 \cic{1}_{ Q_k},f_2\cic{1}_{  Q_\ell})| \leq \varrho |Q_k| \l f_1 \r_{ Q_k} \l f_2 \r_{ Q_\ell}+ \big|\mathbb S^\varrho_{\{R:Q_k, Q_\ell \subset R, \ell(R) \geq 2^{\varrho} \ell(Q_\ell)\} }(\l f_1 \r_{ Q_k}\cic{1}_{Q_k}, \l f_2 \r_{ Q_\ell} \cic{1}_{Q_\ell})\big| \\ &\leq (\varrho+1) |Q_k|\l f_1 \r_{ Q_k} \l f_2 \r_{ Q_\ell} \leq  2^{d} (\varrho+1) |\overline Q| \l f_1 \r_{ \overline Q} \l f_2 \r_{ \overline Q}
 \end{split}
 \]
 whence
\begin{equation}
\label{shiftpf2}
\sum_{k\neq \ell}
 |\mathbb S^\varrho(f_1 \cic{1}_{ Q_k},f_2\cic{1}_{  Q_\ell})| \leq 2^{2d} (\varrho+1) |\overline Q|\l f_1 \r_{ \overline Q} \l f_2 \r_{ \overline Q}.
 \end{equation} We are left with estimating the terms with $k=\ell$ in \eqref{thesum}. We apply  Lemma \ref{mainiter} recursively starting from $Q=Q_{k}$. The recursion stops at the $n$-th step, where $n$ is such that  $S^\varrho_{\mathcal D(Q)}=0$ for all  $Q\in \mathcal S_{\D,k,n}   $.   Such an $n$ exists because of assumption A1. We have
\[
\begin{split} &\quad 
|\mathbb S^\varrho(f_1 \cic{1}_{ Q_k},f_2\cic{1}_{  Q_k})| \leq C\varrho|Q_k| \l f_1 \r_{ Q_k} \l f_2 \r_{ Q_k}+ \sum_{Q\in \mathcal S_{\D,k,1}}  \big|\mathbb S^\varrho_{\mathcal D(Q)}(f_1\cic{1}_Q ,f_2\cic{1}_Q )\big|  \\ & \leq C \varrho|Q_k| \l f_1 \r_{ Q_k} \l f_2 \r_{ Q_k} +  C\varrho\sum_{I\in \mathcal S_{\D,k,1}} |I| \l f_1 \r_{ I} \l f_2 \r_{ I}  +  \sum_{Q\in \mathcal S_{\D,k,2}}  \big|\mathbb S^\varrho_{\mathcal D(Q)}(f_1\cic{1}_Q ,f_2\cic{1}_Q ) \big|
 \\ & \leq \cdots   \leq C\varrho|Q_k| \l f_1 \r_{ Q_k} \l f_2 \r_{ Q_k}+ C\varrho\sum_{Q\in \mathcal S_{\D,k}}   |Q| \l f_1 \r_{ Q} \l f_2 \r_{ Q}\end{split}
 \]
Summing over $k$, and recalling \eqref{shiftpf2}, we obtain that \eqref{thesum} is bounded by the right hand side of \eqref{unifest}, with $\mathcal S_\D$ constructed in the previous subsection, as claimed. Theorem \ref{domshift} is established, up to the proof of Lemma \ref{mainiter}.
\subsection{Proof of Lemma \ref{mainiter}}We set
\[
  E:= \bigcup_{I \in \mathcal I_Q} I, \qquad  \mathcal G:=\left\{R\in \mathcal{D}: R\not\subset E\right\}. \]
  Then 
\[ \mathbb S^\varrho(f_1\cic{1}_Q,f_2\cic{1}_Q)= \mathbb S_{\mathcal G}^\varrho(f_1\cic{1}_Q,f_2\cic{1}_Q)+
\sum_{I \in \mathcal I_Q}\mathbb S^\varrho_{\mathcal D(I)}(f_1\cic{1}_I ,f_2\cic{1}_I ) 
\]
We further decompose $\G$ into $\varrho$ subcollections $\mathcal G'$ such that the sidelengths of $R\in \mathcal G'$ are of the form $2^{\varrho n+m}$ for a fixed $m=0,\ldots,\varrho-1$ and for  some integer $n$. Therefore it suffices to prove that
\begin{equation}
\label{Gprimeest}|\mathbb S_{\mathcal G'}^\varrho(f_1\cic{1}_Q,f_2\cic{1}_Q)| \leq C|Q|\l f_1 \r_Q \l f_2\r_Q\end{equation}
for each of these subcollections $\G'$. To do so, we apply the Calder\'on-Zygmund decomposition to $f_j\cic{1}_Q,$ $j=1,2,$ based on the  collection of disjoint cubes $I\in \mathcal I_Q$. By the stopping condition \eqref{stop} and maximality, these cubes have the property that
\[
\l f_j\r_{I} \leq C \l f_j\r_{Q}, \qquad j=1,2.
\]
Therefore, denoting by $f_{jI}$ the average of $f_j$ on $I$ and setting
\begin{equation} \label{CZdec}
f_j\cic{1}_{Q}=g_j+b_j :=\left(f_j \cic{1}_{E^c} + \sum_{I \in \mathcal I_Q} f_{jI} \cic{1}_I\right) + \left(  \sum_{I \in \mathcal I_Q}b_{jI}\right), \qquad b_{jI}:= (f_j- f_{jI}) \cic{1}_I
\end{equation}
we have for $j=1,2,$
\begin{align}
\label{CZ1}& \|g_j\|_{\infty}\leq   C \l f_j\r_{Q}, \\
\label{CZ2}& \|g_j\|_{2}\leq   C |Q|^{\frac12}\l f_j\r_{Q}, 
\\
\label{CZ3}& \|b_{jI}\|_{1}\leq   C |I|\l f_j\r_{Q}, \qquad I \in \mathcal I_Q.
\end{align}
We need to estimate three types of contributions:
\begin{equation}
\label{Gprimeest2}|\mathbb S_{\mathcal G'}^\varrho(f_1\cic{1}_Q,f_2\cic{1}_Q)| \leq 
|\mathbb S_{\mathcal G'}^\varrho(g_1 ,g_2 )|+ |\mathbb S_{\mathcal G'}^\varrho(g_1 ,b_2 )|+ |\mathbb S_{\mathcal G'}^\varrho(b_1 ,g_2 )| + |\mathbb S_{\mathcal G'}^\varrho(b_1 ,b_2)|
\end{equation}
The $L^2$-boundedness and \eqref{CZ2} yield immediately that
\begin{equation}
\label{Gprimeest3} 
|\mathbb S_{\mathcal G'}^\varrho(g_1 ,g_2 ) |\leq  \|g_1\|_{2}\|g_2\|_{2} \leq C |Q| \l f_1\r_{Q}\l f_2\r_{Q}.
\end{equation}
The second and third summand in \eqref{Gprimeest2} are estimated symmetrically. Considering for instance the second summand, we split
\begin{equation}
\label{Gprimeest4}
|\mathbb S_{\mathcal G'}^\varrho(g_1 ,b_2 )|\leq \sum_{I\in \mathcal I_Q} |\mathbb S_{\mathcal G'}^\varrho(g_1 ,b_{2I} )| \leq \sum_{I\in \mathcal I_Q} \sum_{R\supsetneq I} |S_R (g_1 ,b_{2I} )| \end{equation}
Now when $R\supset I$ and $\ell(R)\geq 2^{\varrho} \ell (I)$, the kernel $s_R(x,\cdot)$ is constant on $I$ and $b_{2I}$ has zero average, whence  $S_R(g_1 ,b_{2I} )=0$. Thus $S_R(g_1 ,b_{2I} )=0$ unless $R=R(I)$, the unique $R\supsetneq I$ with $\ell(R) <2^{\varrho} \ell(I)$. In this case,
we estimate, using the normalization of $s_R$, \eqref{CZ1}, \eqref{CZ3} 
\[
|S_R(g_1 ,b_{2I} )|\leq \|g_1\|_\infty \|b_{2I}\|_1 \leq C |I| \l f_1\r_{Q}\l f_2\r_{Q}.
\]
Now, summing over $I$ in \eqref{Gprimeest4} yields that  the second summand in \eqref{Gprimeest2} is also bounded by $C |Q| \l f_1\r_{Q}\l f_2\r_{Q}$.
We are left with estimating the fourth summand in \eqref{Gprimeest2}. Let $I_1,I_2\in \mathcal I_Q$. Then, reasoning as previously done for $S_R(g_1 ,b_{2I} )=0$, we notice that
$S_R(b_{1I_1} ,b_{2I_2} )=0$ unless $R=R(I_1)=R(I_2)$. Preliminarily observe that the intervals $\{I\in \mathcal I_Q: R(I)=R \}$ are pairwise disjoint and contained in $R$, so that
\begin{equation}
\label{theabove}
\sum_{I: R(I)=R} \|b_{jI}\|_{1} \leq 2 \sum_{I: R(I_1)=R} \|f_j\cic{1}_{I}\|_{1} \leq 2|R|\l f_j \r_R \leq C|R|\l f_j \r_Q,
\end{equation}
where the last inequality follows from $R\not\subset E.$ Therefore 
\begin{equation}
\label{Gprimeest5}
\begin{split} & \quad |\mathbb S_{\mathcal G'}^\varrho(b_1 ,b_2)| \leq 
\sum_{ R \in \mathcal G'}|S_R(b_{1} ,b_{2} )|\leq \sum_{ R \in \mathcal G'} \sum_{I_2: R(I_2)=R} \Big| S_R\left(\sum_{I_1: R(I_1)=R} b_{1I_1} ,b_{2I_2} \right)\Big| \\ & \leq  \sum_{ R \in \mathcal G'} \sum_{I_2: R(I_2)=R}\|b_{2I_2}\|_1 \frac{1}{|R|}  \sum_{I_1: R(I_1)=R} \|b_{1I_1}\|_{1}  \leq C  \l f_1 \r_Q \sum_{ R \in \mathcal G'} \sum_{I_2: R(I_2)=R}\|b_{2I_2}\|_1\\ & \leq C  \l f_1 \r_Q   \sum_{I_2\in \mathcal I_Q }\|f_2\cic{1}_{I_2}\|_1 \leq C|Q|\l f_1 \r_Q   \l f_2 \r_Q,   
\end{split}
\end{equation}
and the fourth summand in  \eqref{Gprimeest2} is also estimated. We have used \eqref{theabove}
to get the second inequality in the second line of the above display,  and the fact that $\{I_2: R(I_2)=R\}$ are disjoint collections over $R\in \mathcal G'$ to pass to the last line. Collecting \eqref{Gprimeest3}, \eqref{Gprimeest4} and \eqref{Gprimeest5} proves the bound \eqref{Gprimeest} and completes the proof of Lemma \ref{mainiter}, and in turn of Theorem \ref{domshift}.

\

\section{The vector-valued case and a matrix $A_2$ bound} \label{s3}

In this section, we extend Theorem \ref{domshift}, as well as its corollaries, to the case of functions taking values in a finite-dimensional Euclidean space $\mathbb F^{n}$. We restrict ourselves to $\mathbb F=\R$ as  the  $\mathbb C^{n}$-valued case  can be easily recovered from the $\R^{2n}$-valued one. 

We begin with defining a handy replacement for the local average of a nonnegative scalar valued function. For a cube $Q\subset \R^d,$ we set $\Phi(Q)=\{\varphi: \R^d\to \R, \|\varphi\|_\infty\leq 1, \supp \phi \subset Q\}$. Then, for $f\in L^1_{\mathrm{loc}}(\R^d; \R^n)$,  set 
\begin{equation} \label{convexbody}
\l f \r_Q:= \left\{\frac{1}{|Q|}\int f \varphi \, \d x: \varphi\in \Phi(Q)  \right\}\subset \R^n.
\end{equation}
It is not hard to see that $\l f \r_Q$ is a closed\footnote{The unit ball of $L^\infty(Q)$ is weak-* compact.} convex symmetric (that is, invariant under reflection through the origin) set. It is also not hard to see that
\[
\sup_{v\in \l f \r_Q } |v| \leq \l |f| \r_Q
\]
where the right hand side simply stands for the average of the scalar function $|f|$ on $Q$.
We have overloaded the notation since the two concepts are essentially the same if the input is scalar.
If $K, H$ are closed convex symmetric sets then their Minkowski product
\[
\{kh:k\in K,h\in H\}
\] 
is a closed symmetric interval and $KH$ denotes indifferently the above interval or its right endpoint. With the above notation, we obtain the following uniform domination theorem for the bilinear forms\footnote{Here  $f_{1,j}$ stands for the $j$-th coordinate of the function $f_{1}$.}
\[
\mathbb S^\varrho \otimes \mathrm{Id}_{\R^n} (f_1,f_2) = \sum_{j=1}^n \mathbb S^\varrho(f_{1,j},f_{2,j}).
\]
\begin{theorem}  \label{domshiftvec} There exists a  constant $C$ depending only on the dimensions $d,n$ such that the following holds. For every $f_1,f_2 \in L^1(\R^d;\R^n)$  with compact support and every dyadic lattice $\mathcal D$  there exists a sparse collection $\mathcal S_\D$ such that\begin{equation}
\label{unifestvec}
\sup_{   \mathbb S^\varrho} |\mathbb S^\varrho \otimes \mathrm{Id}_{\R^n} (f_1,f_2)| \leq C \varrho    \sum_{Q \in \mathcal S_\D} |Q| \l f_1 \r_{Q} \l f_2 \r_{Q}  
\end{equation}
the supremum being taken  over all  {dyadic shifts} $ \mathbb S^\varrho  $ of complexity $\varrho$ constructed on $\mathcal D$.\end{theorem} 
We now consider the class of singular integral operators $\bar{\mathcal U}=\{\mathfrak{Re}\, T: T \in \mathcal U\}$ where the class $\mathcal U$  has been  defined in Section \ref{s1}, and their canonical  extensions to $\R^n$-valued functions,
which we assume defined on the dense subspace   $\mathcal W^n\subset L^2(\R^d;\R^n)$.
As for the scalar case, Proposition \ref{rt} applied componentwise allows us to extend the uniform domination principle of Theorem \ref{domshiftvec} to such family of singular integral operators.
\begin{theorem} \label{mainthmvec}  There exists a  constant $C$ depending only on the dimensions $d,n$ such that the following holds.  For every $f_1,f_2 \in \mathcal W^n$ and having compact support  there exists a sparse collection   $\mathcal S$ such that
\[
\sup_{T\in \bar{\mathcal U}} |\l   T\otimes \mathrm{Id}_{\R^n}( f_1),f_2\r| \leq C  \sum_{Q \in \mathcal S} |Q| \l f_1 \r_{Q} \l f_2 \r_{Q}.  
\]
\end{theorem}
\begin{remark} We were introduced to  definition \eqref{convexbody} in the context of   sparse domination  by S.\ Treil, who, jointly with Nazarov, Petermichl and Volberg, announced the following  result (\cite{NTVc}, with full proof appearing in \cite{NPTV}): for  a standard real-valued Calder\'on-Zygmund kernel operator and for  $f$ in a suitable dense subspace of $L^2(\R^d;\R^n)$ there exist $2^d$ sparse collections $\mathcal S_1,...,\mathcal S_{2^d}$ such that
\begin{equation} \label{NTVdom}
T\otimes \mathrm{Id}_{\R^n}(f)(x) \in  \sum_
{k=1}^{2^d}\sum_{Q \in \mathcal S_k} C\l f \r_{Q} \cic{1}_Q(x)
\end{equation}
almost every $x\in \R^d$
where the summation symbols stand for Minkowski sum. We are not aware of the details of their proof at the time of writing; however,  our result can be suitably interpreted as the dual form of their theorem. 
\end{remark}
 
Finally, we detail an application of Theorem \ref{mainthmvec} to matrix weighted bounds. We say that   $W\in L^1_{\mathrm{loc}}(\R^d; \mathbb{M}_{n,n}(\R))$ is a \emph{matrix weight} if  it is positive semidefinite almost everywhere. We say that a matrix weight $W$ belongs to the class $A_2$ if 
\begin{equation*}
[W]_{A_2}:=\sup_{Q}\|\l W\r_Q^{\frac{1}{2}}\l W^{-1}\r_Q^\frac{1}{2}\|^2<\infty.
\end{equation*}
The following estimate on the weighted space $L^2(W)$, with norm
\[
\|f\|_{L^2(W)}^2=\int_{\R^d} |W^{\frac12}(x)f(x)|^2\, \d x
\]
is a rather immediate consequence of the domination Theorem \ref{mainthmvec} and    of the matrix Carleson embedding theorem of Treil and Volberg \cite{TV}. The derivation from Theorem \ref{mainthmvec} borrows from the approach of Bickel and Wick \cite{BW} to the analogous estimate for $\R^n$-valued sparse averaging operators.
\begin{corollary}\label{corweight}
Let $W\in L^1_{\mathrm{loc}}(\R^d; \mathbb{M}_{n,n}(\R))$ be an  $A_2$ matrix weight. Then there holds
\[ \sup_{T\in\bar{\mathcal U}} \left\| T\otimes \mathrm{Id}_{\R^n}  \right\|_{{L^2(W)}\rightarrow {L^2(W)}} \leq C [W]^{\frac32}_{A_2} \]
where the positive constant $C$ depends on the dimensions $n,d$ only.
\end{corollary}
\begin{remark}
\label{sharpA2} { Previous partial results on sharp dependence of the $L^2$ weighted operator norms of $T $ on the matrix $A_2$ characteristic, as well as  related work on matrix two-weight inequalities, can be found in \cite{BCTW,BPW,BW0} and references therein. The $3/2$ power in our estimate is currently the best known, but,   unlike the scalar case, we have no indication of it being sharp.  The previously mentioned domination theorem \eqref{NTVdom} announced by Nazarov, Petermichl, Treil and Volberg \cite{NTVc} yields the same $3/2$ power; see \cite{NPTV} for details.} \end{remark}

After a preliminary convex set-valued Calder\'on-Zygmund lemma in the upcoming Subsection \ref{ss31}, we detail the proof of Theorem \ref{domshiftvec} in Subsection \ref{ss32} and the derivation of the weighted Corollary \ref{corweight} in the concluding Subsection \ref{ss33}.
\subsection{A convex set-valued Calder\'on-Zygmund lemma} \label{ss31} Before the actual proof of Theorem \ref{domshiftvec}, we need an analogue of the Calder\'on-Zygmund lemma based on the convex sets $\l f \r_Q$.
   \begin{lemma} \label{vectorstop}
Let $A>n^2$, $f\in L^1_{\mathrm{loc}}(\R^d;\R^n)$ and $Q\subset \R^d$ be a cube.  Then the collection $\mathcal I_{Q,f}$ of the maximal dyadic subcubes of $Q$ with
\[
\l f\r_I \not\subset A \l f\r_Q
\]
has the following properties:
\begin{align} \label{convuno}
& \l f\r_I \subset2^d A \l f\r_Q,\\ \label{convdue}
&\displaystyle \sum_{I \in \mathcal I_{Q,f}} |I|< \frac{n^2}{A} |Q|. \end{align}
\end{lemma}
In the proof, we will use the notion of \textit{John ellipsoid}. If $K\subset \R^n$ is a closed convex symmetric set then $\mathcal E_K$, the John ellipsoid, is the     solid ellipsoid of largest volume contained in $K$. This is a closed set with the  property that
\[
\mathcal E_K \subset  K \subset \sqrt{n}\mathcal E_K
\]
where as usual the above denotes concentric dilation. 
\begin{proof}[Proof of Lemma \ref{vectorstop}] We first prove \eqref{convuno} which is rather immediate. It is easy to see that if $\widetilde I $ is the dyadic parent of $I$ then $\l f \r_I\subset 2^d \l f \r_{\widetilde I} $ and the latter set is contained in $2^{d}A \l f\r_Q$ by maximality of $I$.

We come to the proof of \eqref{convdue}. Here we notice that the collection $\mathcal I_{Q,f}$ is invariant under action of $\mathrm{GL}_n(\R)$. For this reason, there is no loss in generality with assuming that the John ellipsoid of $\l f\r_Q $ is the closed unit ball $B$. 
We say that $I\in \mathcal I_{Q,f}$ is of type $j$, $j=1,\ldots, n$ if there exists $F_I\in \l f \r_I $ with $\sqrt{n}(F_{I})_j>A$: here and below $(F_{I})_j$ is the $j$-th coordinate.
Since $\l f \r_I \not \subset A \l f \r_Q $, and \emph{a fortiori} $\l f \r_I \not \subset AB $, it  follows that each $I$ is of type $j$ for at least one $j=1,\ldots,n$. Let $\mathcal I_j$ be those $I\in \mathcal I_{Q,f}$ of type $j$. We will prove that  
\begin{equation} \label{convpf1}
\sum_{I \in \mathcal I_j} |I| < \frac{  n }{A} |Q|
\end{equation}
which in light of the previous observations yields \eqref{convdue}.
We may  find $\varphi_I\in \Phi(I)$   such that 
\[
F_I=\frac{1}{|I|}\int f \varphi_I \, \d x.
\] 
Define now
\[
F_Q:= \sum_{I \in \mathcal I_j} \frac{|I|}{|Q|} F_I   = \frac{1}{|Q|} \int f  \varphi_Q, \qquad \varphi_Q:=\sum_{I \in \mathcal I_j} \varphi_I.
\] 
Since $I$ are pairwise disjoint and contained in $Q$, $\varphi_Q\in \Phi(Q)$.   This means that $F_Q\in \l f \r_Q\subset \sqrt{n} B$. In particular $(F_Q)_j
\leq \sqrt{n}$. But then, applying the   type $j$ condition in the last step
\[
\sqrt{n}\geq  (F_Q)_j =  \sum_{I \in \mathcal I_j} \frac{|I|}{|Q|} (F_I)_j > \frac{A}{\sqrt{n}}
\sum_{I \in \mathcal I_j}\frac{|I|}{|Q|} 
\]
which is \eqref{convpf1}. The proof is thus complete.
\end{proof}

\subsection{Proof of Theorem \ref{domshiftvec}} \label{ss32}
The proof of the vector-valued version follows the  exact same outline of the proof of Theorem \ref{domshift}. We just detail the
 main iterative step, which is carried out through a vector-valued version of Lemma \ref{mainiter}. 
 \begin{lemma} \label{mainitervec} Let $f_1,f_2\in  L^1_{\mathrm{loc}}(\R^d;\R^n)$ and  $Q\in \mathcal D$.  Let
 $\mathcal I_Q$ be the collection of  maximal elements of $\mathcal I_{Q,f_1}\cup \mathcal I_{Q,f_2}. $ Then, for all dyadic shifts  $\mathbb S^\varrho$ constructed on $\mathcal D$
\[
|\mathbb S^\varrho \otimes \mathrm{Id}_{\R^n} (f_1\cic{1}_Q,f_2\cic{1}_Q)| \leq C \varrho |Q| \l f_1 \r_Q \l f_2 \r_Q + \sum_{I \in \mathcal I_Q} \big|\mathbb S^\varrho_{\mathcal D(I)} \otimes \mathrm{Id}_{\R^n}(f_1\cic{1}_I ,f_2\cic{1}_I )\big| 
\]
where $C$ is a positive absolute constant depending on the dimensions $n,d$ only.
\end{lemma}
We clarify that, in the context of functions $f_1,f_2\in  L^1_{\mathrm{loc}}(\R^d;\R^n)$,   the collection $\mathcal I_{Q,f_j}$ refers to the one defined in Lemma \ref{vectorstop} for the value $A=2^{8}n^2$. Thus $\mathcal I_Q$ defined in Lemma \ref{mainitervec} above satisfy \eqref{sparsepack} just like in the scalar case. 
\begin{proof}[Proof of Lemma \ref{mainitervec}]  In this proof, the constant $C$ is meant to depend on $n,d$ only and may vary between instances.

We limit ourselves to indicating the necessary changes from the argument for Lemma \ref{mainiter}. 
We start by setting
\[
  E:= \bigcup_{I \in \mathcal I_Q} I, \qquad  \mathcal G:=\left\{R\in \mathcal{D}: R\not\subset E\right\}. \]
  Then similarly to  the scalar case,
\[ \mathbb S^\varrho\otimes \mathrm{Id}_{\R^n}(f_1\cic{1}_Q,f_2\cic{1}_Q)= \mathbb S_{\mathcal G}^\varrho\otimes \mathrm{Id}_{\R^n}(f_1\cic{1}_Q,f_2\cic{1}_Q)+
\sum_{I \in \mathcal I_Q}\mathbb S^\varrho_{\mathcal D(I)}\otimes \mathrm{Id}_{\R^n}(f_1\cic{1}_I ,f_2\cic{1}_I ). 
\] It thus suffices to prove that
\begin{equation}\label{goodpart}
\left|\mathbb S_{\mathcal G}^\varrho\otimes \mathrm{Id}_{\R^n}(f_1\cic{1}_Q,f_2\cic{1}_Q)\right| \leq C\varrho |Q|\l f_1 \r_Q \l f_2\r_Q.\end{equation}
To prove \eqref{goodpart}, it 
is useful to transform the John ellipsoids of $\l f_1 \r_Q, \l f_2 \r_Q $ to the closed unit ball $B$, which can be achieved via actions of $\mathrm{GL}_{n}(\R)$. For fixed $f_1,f_2$, there exists matrices $A_1,A_2 \in\mathrm{GL}_{n}(\R)$ such that for $j=1,2$, $A_j \tilde{f}_j=f_j$ and the John ellipsoid of $\l \tilde{f}_j \r_Q$ is $B$.  We claim that 
\begin{equation}\label{coordwise}
\left|\mathbb S_{\mathcal G}^\varrho(\tilde{f}_{1,k_1}\cic{1}_Q,\tilde{f}_{2,k_2}\cic{1}_Q)\right|\leq C\varrho |Q|,\qquad \forall k_1,k_2=1,\ldots,n.
\end{equation}
Assuming (\ref{coordwise}), let us first explain how it implies (\ref{goodpart}) and thus the result of the lemma. A simple calculation shows that
\[
\begin{split}
&\quad \left|\mathbb S_{\mathcal G}^\varrho\otimes \mathrm{Id}_{\R^n}(f_1\cic{1}_Q,f_2\cic{1}_Q)\right|
=\left|\mathbb S_{\mathcal G}^\varrho\otimes \mathrm{Id}_{\R^n}\left((A_1\tilde{f}_1)\cic{1}_Q,(A_2\tilde{f}_2)\cic{1}_Q\right)\right| \\
&=\left|\sum_{j,k_1,k_2=1}^nA_1^{jk_1}A_2^{jk_2}\mathbb{S}_{\mathcal{G}}^\varrho (\tilde{f}_{1,k_1}\cic{1}_Q,\tilde{f}_{2,k_2}\cic{1}_Q)\right|\leq C\varrho|Q|\sup_{1\leq k_1,k_2\leq n} \left| \sum_{j=1}^nA_1^{jk_1}A_2^{jk_2} \right|.
\end{split}
\]  We have used  (\ref{coordwise}) in the last  step. Then \eqref{goodpart} will follow if we show that
\begin{equation} \label{Johnkl}
 \left| \sum_{j=1}^nA_1^{jk_1}A_2^{jk_2} \right|
=\left|(A_1 {e}_{k_1})(A_2 {e}_{k_2})\right| \leq \l f_1\r_Q\l f_2\r_Q \qquad \forall 1\leq k_1,k_2\leq n
\end{equation}
where ${e}_k$ is the $k$-th coordinate vector.
Fix $1\leq k_1,k_2\leq n$. By virtue of  $ {e}_{k_j}\in B\subset  \l \tilde{f}_j\r_Q$,   there exists $\varphi_j\in\Phi(Q)$   such that 
\[ {e}_{k_j}= \frac{1}{|Q|}\int \tilde{f}_j\varphi_j, \qquad j=1,2.\]
 Therefore,
\[
A_j {e}_{k_j}=\frac{1}{|Q|}\int (A_j\tilde{f}_j)\varphi_j=\frac{1}{|Q|}\int f_j\varphi_j\in  \l f_j\r_Q,
\] which implies immediately  
the claimed \eqref{Johnkl} and hence (\ref{goodpart}).

Now we turn to the proof of (\ref{coordwise}). We operate the same Calder\'on-Zygmund decomposition of $\tilde{f}_{1,k_j}\cic{1}_Q$  $j=1,2$ as \eqref{CZdec}, this time with respect to the cubes $\mathcal I_Q$ defined in this context. 
The analogues of \eqref{CZ1} and \eqref{CZ3} are, for $j=1,2$
\begin{align}
\label{CZ1v}&  \tilde{g}_j(x) \in    C \l \tilde{f}_j\r_{Q}\subset CB \qquad \forall x \in Q, \\
\label{CZ3v}& \l \tilde{b}_{jI}\r_{I} \subset   C  \l \tilde{f}_j\r_{Q} \subset C  B\qquad  \forall I \in \mathcal I_Q.
\end{align}
An immediate consequence of \eqref{CZ1v}, \eqref{CZ3v} is that  each coordinate  of $\tilde{g}_j,\tilde{b}_j$ satisfies   scalar estimates analogous to \eqref{CZ2} and \eqref{CZ3}: for $j=1,2$ and $1\leq k_j\leq n$
\[
\|\tilde{g}_{j,k_j}\|_2\leq C|Q|^{\frac12}, \quad \|\tilde{b}_{jI,k_j}\|_1\leq C|I|.
\]
By virtue of these estimates and of the fact that the average of each coordinate of $\tilde{b}_{jI}$ vanishes on $I$, the estimate (\ref{coordwise}) follows by repeating the proof of  \eqref{Gprimeest2} from the scalar case.
\end{proof}

\subsection{Proof of Corollary \ref{corweight}} \label{ss33}
Let $T\in \bar{ \mathcal{U}}$ and $W$ be a $n \times n $  matrix $A_2$ weight on $\R^d$. For convenience we write $V_1=W^{-1}, V_2=W$. Since
\[
\| T\otimes \mathrm{Id}_{\R^n}\|_{L^2(W)\rightarrow L^2(W) } =\sup\left\{ \big|\l T\otimes \mathrm{Id}_{\R^n}(V_1f_1),V_2 f_2\r\big|: \|f_1\|_{L^2(V_1)}=\|f_2\|_{L^2(V_2)}=1\right\}
\]
by virtue of the domination Theorem \ref{mainthmvec} it suffices to show that whenever $\mathcal S$ is a sparse collection and $\|f_1\|_{L^2(V_1)}=\|f_2\|_{L^2(V_2)}=1$ there holds
\begin{equation}\label{weightpf1}
\sum_{Q\in \mathcal S} |Q| \l V_1f_1  \r_Q \l  V_2f_2 \r_Q \leq C [W]_{A_2}^{\frac32}.
\end{equation} 
Fix such a collection $\mathcal S$ and $f_1,f_2$. By definition of $\l\cdot \r_Q$, for each $Q$ we may find $\phi_{jQ}\in \Phi(Q) $ such that
\begin{equation} \label{weightpf4}
\l V_1f_1  \r_Q \l  V_2f_2 \r_Q = F_{1Q} F_{2Q} , \qquad F_{jQ}:=\frac{1}{|Q|}\int V_j f_j \phi_{jQ} \, \d x, \quad j=1,2.
\end{equation}
A similar reduction to the one carried out in \cite[Proof of Theorem 1.4]{BW}  then yields that
\begin{equation}\label{weightpf2}
\sum_{Q\in \mathcal S} |Q| \l V_1f_1  \r_Q \l  V_2f_2 \r_Q \leq [W]_{A_2}^{\frac12} \prod_{j=1,2}
 \left( \sum_{Q \in \mathcal S} |Q| \left| \l V_j\r^{-\frac{1}{2}}_QF_{jQ}  \right|^2 \right)^\frac{1}{2}
\end{equation}
which \emph{de facto} reduces \eqref{weightpf1} to proving that, for $j=1,2$ 
\begin{align}\label{weightpf3}
\sum_{Q \in \mathcal S} |Q| \left| \l V_j\r^{-\frac{1}{2}}_QF_{jQ}  \right|^2  = \sum_{Q}   (A_{jQ}F_{jQ}) F_{jQ} \leq C[V_j]_{A_2},
\end{align}
where we set 
\[ A_{jQ}=\begin{cases} 
      0 & Q\notin \mathcal S \\
      |Q|\l V_j\r^{-1}_Q & Q\in \mathcal S
   \end{cases}
\] and recall that $[V_j]_{A_2}=[W]_{A_2}$.
Note that the matrices $A_{jQ}$ satisfy the following packing condition: for any $R\in \mathcal S$,
\begin{align*}
\frac{1}{|R|}\sum_{Q\subset R}\left\|\left<V_j\right>^\frac{1}{2}_{Q}A_{jQ}\left<V_j\right>^\frac{1}{2}_{Q}\right\|&=\frac{1}{|R|}\sum_{\substack{Q\subset R\\ Q\in\mathcal{S} }}|Q|\left\|\left<V_j\right>^\frac{1}{2}_{Q} \l V_j\r^{-1}_Q\left<V_j\right>^\frac{1}{2}_{Q}\right\|
=\frac{1}{|R|}\sum_{\substack{Q\subset R\\ Q\in\mathcal{S} }}|Q| \leq C
\end{align*}
by virtue of the fact that every sparse collection is Carleson: see \cite{LerNaz2015}.
Estimate \eqref{weightpf3} is then a consequence of the Carleson embedding theorem initially due to Treil and Volberg \cite{TV}, in the form recalled in \cite[Theorem 1.3]{BW}. The linear behavior of the  constant in \cite[Theorem 1.3]{BW} has been first obtained by Isralowitz, Kwon and Pott in \cite[Theorem 1.3]{IKP}. 
We remark that, while in \cite[Theorem 1.3]{IKP} $F_{jQ}$ defined in \eqref{weightpf4}  corresponds to the precise choice $\phi_{jQ}=\cic{1}_Q$, the proof works just as well for any choice $\phi_{jQ} \in \Phi(Q)$. This completes the proof of \eqref{weightpf3} and in turn, of Corollary \ref{corweight}.
 \bibliography{CuliucDiPlinioOuMRL}{}
\bibliographystyle{amsplain}
\end{document}